\documentclass{article}
\usepackage{graphicx}
\usepackage{amsmath}
\allowdisplaybreaks[4]
\newtheorem{theorem}{Theorem}

\newtheorem{proposition}{Proposition}

\newenvironment{proof}{{\sc Proof:}}{~\hfill QED}

\title{Another estimating the absolute value of Mertens function}
\author{Rong Qiang Wei\thanks{
                  College of Earth and Planet Sciences, University of Chinese Academy of Sciences, Beijing, PRC, 100049.
                  e\_mail: wrq1973@ucas.edu.cn}}
\date{}
\begin{document}
\newpage
\maketitle
\begin{abstract}
  Through an inversion approach, we suggest a possible estimation for the absolute value of Mertens function $\vert M(x) \vert$ that $ \left\vert M(x) \right\vert \sim \left[\frac{1}{\pi \sqrt{\varepsilon}(x+\varepsilon)}\right]\sqrt{x}$ (where $x$ is an appropriately large real number, and $\varepsilon$ ($0<\varepsilon<1$) is a small real number which makes $2x+\varepsilon$ to be an integer). For any large $x$, we can always find an $\varepsilon$, so that  $\vert M(x) \vert < \left[\frac{1}{\pi \sqrt{\varepsilon}(x+\varepsilon)}\right]\sqrt{x}$.
\end{abstract}

\ \ \ 
{\bf keywords}

\ \ \ Mertens Function, absolute value, M$\ddot{\mathrm{o}}$bius transform, Hilbert transform


\section{Introduction}

The Mertens function $M(n)$ is defined as the cumulative sum of the M$\ddot{\mathrm{o}}$bius function $\mu(k)$ for all positive integers $n$,

\begin{equation}
M(n)=\sum_{k=1}^{n}\mu(k)
\end{equation}

Sometimes the above definition can be extended to real numbers as follows:

\begin{equation}
M(x)=\sum_{1\le k < x}\mu(k)
\end{equation}

In many applications of $M(x)$, the estimation of its absolute value is very important and of interest. Mertens (1897) concluded that $\vert M(x) \vert <x^{1/2}$, which is known as Mertens conjecture. However, this conjecture was proved false through extensive computations (eg., Odlyzko and te Riele,1985; Pintz, 1987; Saouter and te Riele, 2014). These computations above are based or partially based on sieving, and $\vert M(x)\vert $ has been computed for all $x\le 10^{22} $(Kuznetsov, 2011). The isolated values of $M(2^n)$ has been computed for all positive integers $n\le 73$ ($\vert M(2^{73})\vert=6524408924$; Hurst, 2016). Moreover, Wei (2016) suggested a disproof for Mertens conjecture from the viewpoint of probability.

There are many other explicit upper bounds for $\vert M(x)\vert$. For example, Wal$\acute{}{\rm\ }$fi$\check{s}$ (1963) showed that,

\begin{equation}
\vert M(x)\vert \le x\exp\left[-c \ln^{3/5} x(\ln\ln x)^{-1/5}\right]
\end{equation}
where $c$ is a constant.

MacLeod (1967; 1969) showed that,

\begin{equation}
\left | M(x)\right | \le \frac{x+1}{80}+\frac{11}{2} \hspace{5em} (x\ge 1)
\end{equation}
 
El Marraki (1995) proved that,

\begin{equation}
\left | M(x) \right | \le \frac{0.002969}{(\log x)^{1/2}}x \hspace{5em} (x\ge 142194)
\end{equation}

and

\begin{equation}
\left | M(x) \right | \le \frac{0.6437752}{\log x}x \hspace{5em} (x > 1)
\end{equation}

Ramar$\acute{e}$ (2013) showed that,

\begin{equation}
\left | M(x) \right | \le \frac{0.0146\log x-0.1098}{(\log x)^2} x \hspace{5em} (x \ge 464402)
\end{equation}

Other explicit bound for $\vert M(x)\vert$ are\footnote{From https://en.wikipedia.org/wiki/Mertens\_function.},

\begin{equation}
\left | M(x) \right | <\frac{x}{4345} \hspace{5em} (x > 2160535 )
\end{equation}

\begin{equation}
\left | M(x) \right | <\frac{0.58782x}{\ln^{11/9}(x)} \hspace{5em} (x >685)
\end{equation}

Wei (2016) has discussed $\vert M(x)\vert$ based on the assumption that $\mu (n)$ is an independent random sequence. This assumption is inferred from the numerical consistency between empirical statistical quantities for $2\times 10^7$ $\mu(n)$s and those from number theory. The following inequality (\ref{eq1}) for $\vert M(x)\vert$ holds with a probability of $1 - \alpha$, 

\begin{equation}\label{eq1}
\vert M(x)\vert \le  \frac{\sqrt {6/{\pi ^2}}}{\sqrt{\alpha}} \sqrt{x}
\end{equation}

However, inequality (\ref{eq1}) is argued for the reason that the prime factorization of integers is not random (the primes being intricately interdependent) and so the $\mu(n)$ is not random. Wei (2018) estimated the $\vert M(x)\vert$ by an approach of statistical mechanics, in which the $\mu(n)$ is taken as a particular state of a modified one-dimensional Ising model without the exchange interaction between the spins. If $M(x)$ is a quantity that can be measured, $\vert M(x)\vert \le \sqrt{\frac{C}{\alpha}x}$ (where $C$ is constant) with a probability $1-\alpha$ ($0<\alpha<1$) from the viewpoint of the energy fluctuations in the canonical ensemble. Without the assumption above, Wei (2017) further explained that $\vert M(x)\vert \le C x^{1/2}$ with a very large $C$ (possible $+\infty$) based on three facts.  Recently, Czopik (2019) studied on the estimation of the $M(x)$.

A lot of experiences show that it is difficult to estimate $\vert M(x)\vert$ directly. However, if we can represent $M(x)$ as other functions which can be easily studied, it is possible to estimate $\vert M(x)\vert$.

\section{An estimation for $\vert M(x)\vert$} 
   
\begin{theorem}\label{Theo1}
    Let $f(t)$ be a real-valued and positive function which is monotonic decreasing and integrable on $[a, +\infty)$ (where $a$ is an integer). Then there is a constant $c$, so that for each real number $A\ge c$, we have\footnote{This theorem is not proposed by us. We have forgotten its source and we prove it again with our understanding. Readers who are interested it please refer to Iwaniec and Kowalski (2004), in which there is a similar conclusion (p. 19-20). },
    \begin{equation}
     \left\vert \sum_{_{a\le n\le A}} f(n)-\left[c+\int_{a}^A f(t){\rm d} t\right] \right\vert \le f(A)  
   \end{equation}
 \end{theorem}

\begin{proof}
     For the integer $n\ge a$, let
     \begin{equation}
     b_n=f(n)-\int_{n}^{n+1}f(t){\rm d}t   
   \end{equation}
   
   Because $f(t)$ is monotonic decreasing and positive, and there is a point $\xi \in (n,n+1)$ at least at which $\int_{n}^{n+1}f(t){\rm d}t=f(\xi)$, we have,
     \begin{equation}
     f(n)\ge \int_{n}^{n+1}f(t){\rm d}t \ge f(n+1) 
   \end{equation}
   
    Further, we have,
     \begin{equation}
     f(n)-f(n+1)\ge b_n\ge 0 
   \end{equation}
   
   Let $s_{_N}=\sum_{n=a}^N b_n$ (where $N=a, a+1, a+2, ....$). Then,
    \begin{equation}\label{sN}
     s_{_N}\le \sum_{n=a}^N[f(n)-f(n+1)]=f(a)-f(N+1)<f(a) 
   \end{equation}
   
   (\ref{sN}) shows that $s_{_N}$ is an increasing sequence ($s_{_N}=f(a)-f(N+1)<s_{_{N+1}}=f(a)-f(N+2)<f(a)$) but it has an upper bound, so it has a limit,
    \begin{equation}\label{on_c}
     \lim_{_{N\rightarrow +\infty}} s_{_N}=c < f(a)
    \end{equation} 

    Further, we have,
     \begin{equation}
     0<(c-\sum_{a\le n \le A} b_n)=\lim_{_{N\rightarrow +\infty}} \sum_{n=[A]+1}^N b_n\le f([A]+1)\le f(A)
    \end{equation}    
    
    Then,
      \begin{equation}\label{cfA}
     c-f(A)\le \sum_{a\le n \le A} b_n\le c
    \end{equation}
    
    Notice that
   \begin{align}
    \begin{aligned}
 \sum_{a\le n \le A} b_n & =\sum_{a\le n \le A} f(n)-\sum_{a\le n \le A}\int_n^{n+1} f(t)\rm d t\\
  \ &=\sum_{a\le n \le A} f(n)-\int_a^{[A]+1} f(t){\rm d} t
   \end{aligned}
    \end{align}
       
    From (\ref{cfA}) we have,
    \begin{equation}
    c-f(A) \le  \sum_{a\le n \le A} f(n)-\int_a^{[A]+1} f(t){\rm d} t \le c
    \end{equation}
    
    That is,
     \begin{equation}\label{cfAc}
    c-f(A)+\int_a^{[A]+1} f(t){\rm d} t \le  \sum_{a\le n \le A} f(n) \le c+\int_a^{[A]+1} f(t){\rm d} t
    \end{equation}
    
    Notice in (\ref{cfAc}) that
    \begin{align}
     \begin{aligned}
    \int_a^{[A]+1} f(t){\rm d} t & = \int_a^A f(t){\rm d} t+\int_A^{[A]+1} f(t){\rm d} t  \\
    \ & = \int_a^A f(t){\rm d} t+ f(\xi) (A<\xi<[A]+1) \\
    \ & \le \int_a^A f(t){\rm d} t+ f(A)
     \end{aligned}
   \end{align}
     
    Thus we know,
     \begin{equation}
    \sum_{a\le n \le A} f(n) \le  c+\int_a^A f(t){\rm d} t+ f(A)
    \end{equation}
    
    And finally,
    \begin{equation}
     \left\vert \sum_{_{a\le n\le A}} f(n)-\left[c+\int_{a}^A f(t){\rm d} t\right] \right\vert \le f(A) 
    \end{equation}
   
\end{proof}

  \begin{proposition}\label{pro1}
   \begin{equation}
   \sum_{_{1\le n \le A}} \frac{1}{n-x} \sim \ln\frac{A}{x+\varepsilon}\ \ \ \ \ \ (A\rightarrow +\infty)
   \end{equation}
   where $x\ne n$ is a real number, $0<\varepsilon<1$, and $2x+\varepsilon$ is an integer.
   \end{proposition}
  
     \begin{proof}
   If we let $f(t)=\frac{1}{t-x}(t\ne x)$, and $a=2x+\varepsilon$ in theorem \ref{Theo1}, we have,
     \begin{align}\label{convergent1}
    \begin{aligned}
    \ & \left\vert \sum_{_{2 x+\varepsilon\le n\le A}} \frac{1}{n-x}-\left[c+\int_{2x+\varepsilon}^A\frac{1}{t-x}{\rm d}t\right]\right\vert \le \frac{1}{A-x}\\
    {\rm i.e. } & \left\vert \sum_{_{2 x+\varepsilon\le n\le A}} \frac{1}{n-x}-\left[c+\ln\frac{A-x}{x+\varepsilon}\right]\right\vert \le \frac{1}{A-x}\\
    {\rm i.e. } & \left\vert \sum_{_{2 x+\varepsilon\le n\le A}} \frac{1}{n-x}-\left[c+\ln\frac{A(1-x/A)}{x+\varepsilon}\right]\right\vert \le \frac{1}{A(1-x/A)}\\
     {\rm i.e. } & \left\vert \sum_{_{2 x+\varepsilon\le n\le A}} \frac{1}{n-x}-\left[c+\ln\frac{A}{x+\varepsilon}\right]\right\vert \le \frac{1}{A} \ \ \ \ (A\rightarrow +\infty)
    \end{aligned}
     \end{align}
    
    The last step of (\ref{convergent1}) shows that $\sum_{_{2 x+\varepsilon\le n\le A}} \frac{1}{n-x}$ converges to $c+\ln\frac{A}{x+\varepsilon}$ (uniformly), ie.,
    
     \begin{align}\label{nx0}
      \begin{aligned}
     \sum_{_{2x+\varepsilon\le n\le A}} \frac{1}{n-x} & = c+ \ln\frac{A}{x+\varepsilon} 
     \end{aligned}
     \end{align}
      
     Adding $d_n=\sum_{_{1\le n\le (2x+\varepsilon) -1}} \frac{1}{n-x}$ to both sides of the (\ref{nx0}) and because $d_n$ is finite even $d_n\rightarrow 0$ if an appropriate $x$ is selected (see details in Remark 3), we can get, 
     \begin{align}\label{nx1}
      \begin{aligned}
    d_n + \sum_{_{2x+\varepsilon\le n\le A}} \frac{1}{n-x} & =  \sum_{_{1\le n \le A}} \frac{1}{n-x}\\
    \ & = c+ \ln\frac{A}{x+\varepsilon}+d_n  \\  
    \ & \sim \ln\frac{A}{x+\varepsilon}   \ \ \ \ \ \ (A\rightarrow +\infty)  
     \end{aligned}
     \end{align}
     
     \end{proof}
     
      \ \ \ \ \ 
    
      {\bf Remark 1.  }

      In proposition \ref{pro1} the reason for introduction of the parameter $\varepsilon$ is that Theorem \ref{Theo1} requires $a$ to be an integer, but $x$ being an integer will cause $\frac{1}{n-x}$ to be meaningless. Therefore, we add an $\varepsilon$ to $2x$ to make $2x+\varepsilon$ ($x$ now is a real number) an integer. Moreover, it can be seen that, like $A$, $\varepsilon$ is hidden on the left side of the proposition \ref{pro1}, so the left and right sides of proposition \ref{pro1} are "balanced" (the same below).
      
     \ \ \ \ \ 
    
      {\bf Remark 2.  }

      In (\ref{nx1}) $c$ is omitted because it is finite and (far) less than the logarithmic part. In fact, from (\ref{on_c}) it can be seen that $c<f(a)=\frac{1}{x+\varepsilon}$. In order to convince oneself this is plausible, it is instructive to consider the Euler constant $\gamma$,

    \begin{equation}
     \gamma=\lim_{N\rightarrow +\infty}=(\sum_{n=1}^N\frac{1}{n}-\ln N)=0.577215...
    \end{equation}
    
    \ \ \ \ \ 
    
      {\bf Remark 3.  }

      In (\ref{nx1}) $d_n\rightarrow 0$ when $x$ is very large, $x=n+\theta/2{\rm \ } (0<\theta< 1)$, and $\theta\rightarrow 1$. This can be seen from the following,
      
      It is well known that,
      
       \begin{align}
    \begin{aligned}
    \psi (x-N) & =\psi(x)-\sum_{n=1}^{N}\frac{1}{x-n}\\
    \ & =\psi(x)+\sum_{k=1}^{N}\frac{1}{n-x}
    \end{aligned}
     \end{align}
     where $\psi(x)$ is Psi function. 
     
     Therefore,
     \begin{align}\label{dn}
    \begin{aligned}
   \sum_{n=1}^{2x+\varepsilon}\frac{1}{n-x} & =\psi(-x-\varepsilon)-\psi(x)\\
    \ & \approx\psi(-x-\varepsilon)-\psi(x+\varepsilon)\\
    \ & =\pi \cot[\pi (x+\varepsilon)]+\frac{1}{x+\varepsilon}\\
    \ & \sim \pi \cot  [\pi (x+\varepsilon)]
    \end{aligned}
     \end{align}
where $\varepsilon$ is very small, $x$ is very large, and the upper limit of the sum is $2x+\varepsilon$ rather than $2x+\varepsilon -1$ for convenience.    
     
    Obviously $\cot  [\pi (x+\varepsilon)]=0$ when $x+\varepsilon=n+1/2$. However, $x\ne n$, and we can let $x=n+\theta/2{\rm \ } (0<\theta< 1)$. So that $\cot  [\pi (x+\varepsilon)]\rightarrow 0$ when $\theta\rightarrow 1$. At this time, $\varepsilon\rightarrow 0$.  Notice that $2x+\varepsilon=2n+\theta +\varepsilon$ is an integer.  
     
\begin{theorem}\label{Theo2}
    Let $f(x)$ and $g(x)$ be any functions. If 
    \begin{equation}\label{chen_t1}
     g(x)=\sum_{k=1}^\infty f(k-x)  
   \end{equation}
   then
   \begin{equation}\label{chen_t2}
     f(x)=\sum_{k=1}^\infty \mu(k) g(k-x)
   \end{equation}
   \end{theorem}

   \begin{proof}
   Please see Theorem 2 and the Propositions related in Chen (1997). (\ref{chen_t1}) and (\ref{chen_t2}) are special cases of this theorem which are named the parametric M$\ddot{\mathrm{o}}$bius transform formulas by Chen (1997). Or please refer to Theorem 1 in Knockaert (1994). It should be pointed out that the "absolutely convergent" in Theorem 2 of Chen (1997) means that "every infinite series converges to an element which is independent of the combinations of the tems in the sum". 
   \end{proof}
   
   \begin{theorem}\label{Theo3}
    \begin{equation}
     M(x) \sim -\frac{\sqrt{x}}{\pi^2}\int_0^\infty \frac{{\rm d}t}{(t+\varepsilon)\sqrt{t}(t-x)}  
   \end{equation}
  \end{theorem}

   \begin{proof}
    
    From (\ref{nx1}), we can see that (\ref{chen1}) is true, ie., $\sum_{1\le n \le A}\frac{1}{n-x}$ converges to $\ln\frac{A}{x+\varepsilon}$ (uniformly) with the increasing $A$ to $+\infty$,
    
   \begin{equation}\label{chen1}
   \ln\frac{A}{x+\varepsilon} \sim \sum_{1\le n \le A}\frac{1}{n-x} \ \ \ \ \ \ (A\rightarrow +\infty)
   \end{equation}

   Thus from Theorem \ref{Theo2} we can write out (\ref{chen2}),
   
   \begin{equation}\label{chen2}
  \frac{1}{x} \sim \sum_{1\le n \le A}\mu(n)\ln\frac{A}{n-(x-\varepsilon)}
  \end{equation}    
   
  Let $x'=x-\varepsilon$, and $x$ denotes $x'$ for simplicity. We have,
     
   \begin{align}
   \begin{aligned}
   \frac{1}{x+\varepsilon}& \sim \sum_{1\le n \le A}\mu(n)\ln\frac{A}{n-x}\\
   \ & = \sum_{1\le n \le A}\left[M(n)-M(n-1)\right]\ln\frac{A}{n-x}\\
   \ & =\sum_{1\le n \le A}\left[M(n)\ln\frac{A}{n-x}-M(n-1)\ln\frac{A}{n-x}\right]\\
   \ & =\sum_{1\le n \le A}\left[M(n)\ln\frac{A}{n-x}-M(n)\ln\frac{A}{(n+1)-x}\right] \\
   \ & =\sum_{1\le n \le A}M(n)\left[\ln\frac{A}{n-x}-\ln\frac{A}{(n+1)-x}\right]\\
   \ & =\sum_{1\le n \le A}M(n)\left[\ln((n+1)-x)-\ln(n-x)\right]\\
   \ & =\sum_{1\le n \le A}M(n)\int_n^{n+1}\frac{{\rm d}t}{t-x} 
   \end{aligned}
   \end{align}

   Noticing that $M(n)$ is constant in each $(n,n+1)$, and if we assume $M(0)=0$, we can get,
    
     \begin{align}\label{hilbert1}
     \begin{aligned}
      \frac{1}{x+\varepsilon}& \sim \sum_{1\le n \le A}M(n)\int_n^{n+1}\frac{{\rm d} t}{t-x}\\
      \ & = \sum_{n=1}^A\int_n^{n+1} M(t)\frac{{\rm d}t}{t-x} \\
     \ & =\int_1^{A}M(t)\frac{{\rm d}t}{t-x}     \\
      \ & =\int_0^{A}M(t)\frac{{\rm d}t}{t-x}  \ \ \ \ \ \ \ (M(0)=0)\\
     \ & =\int_0^{\infty}M(t)\frac{{\rm d}t}{t-x} \ \ \ \ \ \ \ (A\rightarrow +\infty)
     \end{aligned}
   \end{align}
   
   The last step of (\ref{hilbert1}) is Hilbert transform of the $M(t)$ on the semiaxis and its inverse transform (eg., Polyanin and Manzhirov, 2008; Antipov and Mkhitaryan, 2018. See Appendix in detail) is,
   
   \begin{equation}\label{hilbert2}
    M(x) \sim -\frac{\sqrt{x}}{\pi^2}\int_0^\infty \frac{{\rm d}t}{(t+\varepsilon)\sqrt{t}(t-x)}  
    \end{equation}
   \end{proof}
   
    \begin{theorem}\label{Theo4}
    \begin{equation}
    \left\vert M(x) \right\vert \sim \left[\frac{1}{\pi \sqrt{\varepsilon}(x+\varepsilon)}\right]\sqrt{x}  
   \end{equation}
  \end{theorem}
  
  \begin{proof}
  Because
  \begin{equation}\label{jf}
  {\rm PV\ \ }\int_0^\infty\frac{t^{\mu-1}{\rm d}t}{(t+\varepsilon)(x-t)}=\frac{\pi}{x+\varepsilon}\left[\frac{\varepsilon^{\mu -1}}{\sin (\mu\pi)}+x^{\mu -1}\cot(\mu\pi)\right]
  \end{equation}
  where $\vert {\rm arg} \varepsilon\vert <\pi, x>0, 0<{\rm Re}\mu <2$ (Erd$\grave{e}$ly et al., 1954; Gradshteyn and Ryzhik, 2014).
  
  Let $\mu = 1/2 $ in (\ref{jf}). We have,
  
   \begin{equation}\label{jf2}
    -\int_0^\infty\frac{t^{-1/2}{\rm d}t}{(t+\varepsilon)(t-x)}=-\frac{\pi}{\sqrt{\varepsilon}(x+\varepsilon)}
   \end{equation}
   
   From Theorem (\ref{Theo3}), we know,
    \begin{align}
     \begin{aligned}
    \vert M(x)\vert & \sim \left\vert \frac{\sqrt{x}}{\pi^2}\int_0^\infty \frac{{\rm d}t}{(t+\varepsilon)\sqrt{t}(t-x)} \right\vert \\
     \ & = \left[\frac{1}{\pi \sqrt{\varepsilon}(x+\varepsilon)}\right]\sqrt{x}
      \end{aligned}
   \end{align}
    \end{proof}
  
  \section{Discussions}
 \subsection{Similarity of Theorem \ref{Theo4} to (\ref{dis1})}
  We can change the form of Theorem (\ref{Theo4}) to (\ref{dis1}),
   \begin{equation}\label{dis1}
    \left\vert M(x) \right\vert \sim \left[\frac{1}{\pi \sqrt{1-\alpha}(x+1-\alpha)}\right]\sqrt{x}  
   \end{equation}  
  
  Since $\alpha{\rm \ } (0<\alpha<1)$ is not fixed, for a reasonable $x$ ($=n+\theta/2 {\rm \ },\theta\rightarrow 1$), an $\alpha$ can always be found so that the following holds,

\begin{equation}\label{dis2}
\vert M(x) \vert < \left[\frac{1}{\pi \sqrt{1-\alpha}(x+1-\alpha)}\right]\sqrt{x}
\end{equation}

  (\ref{dis2}) is similar to those like (\ref{eq1}) in the introduction, in which $\vert M(x)\vert \le \sqrt{\frac{C}{\alpha}x}$ (where $C$ is constant) holds with a probability $1-\alpha$ ($0<\alpha<1$). The later are estimated based on the assumption that $\mu (n)$ is an independent random sequence, or from the view point of the energy fluctuations in the canonical ensemble of statistical mechanics. This similarity shows that $\mu (n)$ could be taken as an independent random sequence. Although $\mu (n)$ is determined according to that the prime factorization of integers is not random, but it can be produced by an independent and random function. This can be found from another definition of $\mu(n)$,

\begin{equation}\label{dis3}
\mu(n)  = \left\{ {\begin{array}{*{20}{c}}
0\\
{{{( - 1)}^{\omega(n)}}}
\end{array}\begin{array}{*{20}{l}}
{{\mbox{if \ }} n {\mbox{\ is non-squarefree\hspace{0em}}}}\\
{{\mbox{if\  }} n {\mbox{\  is squarefree}}}
\end{array}} \right.
\end{equation}
where $\omega(n)$ is the number of distinct prime factors.  According to Erd$\ddot{\mathrm{o}}$s-Kac Theorem, $\omega(n)$ is independent and random when $n$ is large.

On the other hand, the determined function could produce random distribution, for example, the logistic map (\ref{dis4}) will generate random numbers (Trott, 2004), 

\begin{equation}\label{dis4}
x_{n+1}=\lambda (1-x_n)
\end{equation}
where $x_n\in [0,1]$, $n=1,2,3,...$, $0<\lambda\le 4$.

Therefore, the deterministic problems can be handled by the methods of probability to some extend, or vice versa. For example, some partial differential equations can be studied with the approaches of probability, or an deterministic definite integral can be calculated by the Monte Carlo method.

\subsection{On the Theorem \ref{Theo4}}

It is an Fredholm integral equation of the first kind to obtain $M(x)$ from $\frac{1}{x+\varepsilon}$ in Eq. (\ref{hilbert1}). Such an integral equation is a typical ill-posed problem. Its solution is not unique. Therefore  Theorem \ref{Theo4} is one of the possible solution. In fact, the introduction of free parameters $\varepsilon$ can give a glimpse to this. 

\section{Conclusion}

We suggest a possible estimation to $\vert M(x) \vert$ by the inverse Hilbert Transform on the real semiaxis of $\frac{1}{x+\varepsilon}$ ($0<\varepsilon<1$). For an appropriately $x$ ($=n+\theta/2 {\rm \ },0<\theta<1, \theta\rightarrow 1$), we can always find a positive real number $\varepsilon$ which makes $2x+\varepsilon$ to be an integer, so that  $\vert M(x) \vert \sim \left[\frac{1}{\pi \sqrt{\varepsilon}(x+\varepsilon)}\right]\sqrt{x}$. Further, we can select a reasonable $\varepsilon$, so that $\vert M(x) \vert < \left[\frac{1}{\pi \sqrt{\varepsilon}(x+\varepsilon)}\right]\sqrt{x}$.

\vspace{5em}



{\Large\bf  Appendix}

\ \ \ \ 

Hilbert transform pair on the semiaxis (eg., Polyanin and Manzhirov, 2008; Antipov and Mkhitaryan, 2018) are as follows,

\begin{equation}
f(x)=\int_0^\infty \frac{y(t)}{t-x} {\rm d}t
\end{equation}

\begin{equation}
y(x)=-\frac{\sqrt{x}}{\pi^2}\int_0^\infty \frac{f(t)}{\sqrt{t}(t-x)} {\rm d}t
\end{equation}

\end{document}